\documentclass[12pt,reqno]{amsart}
\usepackage{times}

\newtheorem{thm}{Theorem}
\newtheorem{lem}[thm]{Lemma}

\newtheorem{prop}[thm]{Proposition}
\newtheorem*{thma}{Theorem A}
\newtheorem*{thmb}{Theorem B}
\newtheorem*{thmc}{Corollary C}
\newtheorem*{thmd}{Corollary D}
\newtheorem*{thme}{Corollary E}
\newcommand{\C}{{\mathbb C}}
\newcommand{\cn}{{\mathbb C}^n}
\newcommand{\R}{{\mathbb R}}
\newcommand{\bA}{{\mathbb A}}
\newcommand{\bB}{{\mathbb B}}

\begin{document}

\title{Geometric Spectral Theory\\ for Compact Operators}

\author{Isaak Chagouel, Michael Stessin, and Kehe Zhu}
\address{Department of Mathematics and Statistics\\
         State University of New York\\
         Albany, NY 12222, USA}

\email{ichagouel@albany.edu}
\email{mstessin@albany.edu}
\email{kzhu@albany.edu}


\begin{abstract}
For an $n$-tuple $\bA=(A_1,\cdots,A_n)$ of compact operators we define the
joint point spectrum of $\bA$ to be the set
$$\sigma_p(\bA)=\{(z_1,\cdots,z_n)\in\cn:\ker(I+z_1A_1+\cdots+z_nA_n)\not=(0)\}.$$
We prove in several situations that the operators in $\bA$ pairwise commute
if and only if $\sigma_p(\bA)$ consists of countably many, locally finite,
hyperplanes in $\cn$. In particular, we show that if $\bA$ is an $n$-tuple of
$N\times N$ normal matrices, then these matrices pairwise commute if and only
if the polynomial
$$p_{\bA}(z_1,\cdots,z_n)=\det(I+z_1A_1+\cdots+z_nA_n)$$
is completely reducible, namely,
$$p_{\bA}(z_1,\cdots,z_n)=\prod_{k=1}^N(1+a_{k1}z_1+\cdots+a_{kn}z_n)$$
can be factored into the product of linear polynomials.
\end{abstract}

\maketitle

\section{Introduction}

The theory of single operators is by now a very mature subject, with the notion of
spectrum playing a key role in the theory. However, multivariate operator theory is
only in its very early stages of development. There is not even wide agreement about
how ``the joint spectrum'' of an $n$-tuple
$$\bA=(A_1,\cdots,A_n)$$
of bounded linear operators on the same Hilbert space $H$ should be defined.

The Taylor spectrum is probably the most studied generalization of the notion of
spectrum for a single operator to the setting of several operators. The definition of
the Taylor spectrum must rely on the extra assumption that the tuple $\bA$ consists
of mutually commuting operators. See \cite{T}. Another notion of joint spectrum was
introduced and studied by McIntosh and Pride \cite{MP1, MP2}. It was further investigated in
\cite{P1, P2, PS, R2, S}. In general, this definition did not require mutual commutativity.

A more elementary notion of joint spectrum for an $n$-tuple $\bA$ of operators on $H$ was
recently introduced by Yang in \cite{Y} and further studied in \cite{SYZ}. More
specifically, Yang defines $\Sigma(\bA)$ to be the set of points $z=(z_1,\cdots,z_n)\in\cn$
such that the operator $z_1A_1+\cdots+z_nA_n$ is not invertible. It is clear that
if $z\in\Sigma(\bA)$, then $cz\in\Sigma(\bA)$ for any complex constant $c$. Therefore,
it is more appropriate to think of $\Sigma(\bA)$ as a subset of the complex projective
space $\C P_n$. Because of this, Yang called $\Sigma(\bA)$ the projective spectrum of $\bA$.
The definition of $\Sigma(\bA)$ is straighforward and there is no need to make the
assumption that the operators in $\bA$ commute with each other.

It was recently discovered in \cite{GS,OS} that the projective spectrum plays an important
role in certain extremal problems of numerical analysis. For example, Theorem 2 in \cite{OS}
shows that the simpler the geometry of the projective spectrum is, the easier the solution of
the extremal problem is. In particular, if the projective spectrum consists of the union of
hyperplanes, then the solution of the corresponding extremal problem is the easiest and the
most natural. Thus, it is important to understand how the geometry of the projective spectrum
is connected to the mutual behavior of these operators.

The purpose of this paper is to study the relationship between the mutual commutativity
of operators in $\bA$ and properties of the projective spectrum for an $n$-tuple $\bA$ of
compact operators. In general, the projective spectrum can be non-informative. For example, if
all operators in $\bA=(A_1,\cdots,A_n)$ are compact, the projective spectrum coincides with
the whole $\C P_n$. Such a degeneration cannot occur if at least one of the operators is
invertible. In this case the projective spectrum is a proper subset of $\C P_n$. If one of the
operators, say $A_n$, is invertible, we may assume that it is the identity, since
$\Sigma (A_1,\cdots,A_n)=\Sigma(A_n^{-1}A_1,\cdots, A_n^{-1}A_{n-1},I)$.
For this and other reasons (see next section), it makes sense to append the identity operator to
$\bA$. Our main results show that in many situations the commutativity of
operators in $\bA=(A_1,\cdots,A_n)$ is equivalent to a certain linear structure of the
projective spectrum of the expanded tuple $(A_1,\cdots,A_n,I)$.

In view of the remarks above and to state our main results, we will slightly
modify the notion of the projective spectrum. Thus we define $\sigma(\bA)$ to be the set
of points $z=(z_1,\cdots,z_n)\in\cn$ such that the operator $I+z_1A_1+\cdots+z_nA_n$ is not
invertible. Similarly, we define $\sigma_p(\bA)$ to be the set of points $z\in\cn$ such that the
operator $I+z_1A_1+\cdots+z_nA_n$ has a nontrivial kernel. Throughout the paper we assume that
there is at least one operator in $\bA$ that is nonzero. This will ensure that $\sigma(\bA)$ is
non-empty. In the case of compact operators, this will also ensure that $\sigma_p(\bA)$ is
non-empty. We can now state our main results.

\begin{thma}
Suppose $\bA=(A_1,\cdots,A_n)$ is a tuple of compact, self-adjoint operators on a Hilbert
space $H$. Then the operators in $\bA$ pairwise commute if and only if $\sigma_p(\bA)$
consists of countably many, locally finite, complex hyperplanes in $\cn$.
\end{thma}

Recall from algebra and algebraic geometry that a polynomial is completely reducible if
it can be factored into a product of linear polynomials. A simple example of a polynomial
of two variables that cannot be factored into the product of linear polynomials is
$z^2+w$.

\begin{thmb}
Suppose $\bA=(A_1,\cdots,A_n)$ is a tuple of $N\times N$ normal matrices. Then the
following conditions are equivalent:
\begin{enumerate}
\item[(a)] The matrices in $\bA$ pairwise commute.
\item[(b)] $\sigma_p(\bA)$ is the union of finitely many complex hyperplanes in $\cn$.
\item[(c)] The complex polynomial
$$p(z_1,\cdots,z_N)=\det(I+z_1A_1+\cdots+z_nA_n)$$
is completely reducible.
\end{enumerate}
\end{thmb}

As consequences of Theorem A, we will also obtain the following three corollaries.

\begin{thmc}
A compact operator $A$ is normal if and only if $\sigma_p(A,A^*)$ is the union of
countably many, locally finite, complex lines in $\C^2$.
\end{thmc}

\begin{thmd}
Two compact operators $A$ and $B$ are normal and commute if and only if
$\sigma_p(A,A^*,B,B^*)$ is the union of countably many, locally finite,
complex hyperplanes in $\C^4$.
\end{thmd}

\begin{thme}
Two compact operators $A$ and $B$ commute completely (that is, $A$ commutes with both $B$
and $B^*$) if and only if each of the four sets $\sigma_p(A\pm A^*, B\pm B^*)$ is the union
of countably many, locally finite, complex lines in $\C^2$.
\end{thme}

We will give a simple example of two $2\times2$ matrices $A$ and $B$ such that
$\sigma_p(A,B)$ is the union of two complex lines in $\C^2$, but $AB\not=BA$. This
shows that additional assumptions (such as normality or self-adjointness), other than
compactness, are indeed necessary.

We wish to thank our colleague Rongwei Yang for many useful conversations.

\section{The example of $2\times2$ matrices}

To motivate later discussions and to convince the reader that our main results
are indeed correct, we begin with the case of $2\times2$ matrices. In this case,
we can solve the problem by explicit computation. However, it will be clear that
this direct approach is impossible to extend to higher order matrices, let alone
arbitrary operators. New ideas are needed to tackle the problem for more general
operators, including higher order matrices.

Thus we begin with two normal $2\times2$ matrices $A$ and $B$, and proceed to show
that $AB=BA$ if and only if $\sigma_p(A,B)$ is the union of finitely many complex
lines in $\C^2$ if and only if the {\em characteristic polynomial} of $(A,B)$,
$\det(zA+wB+I)$, can be factored into the product of linear polynomials.

Since $A$ is normal, there exists a unitary matrix $U$ such that $A=U^*DU$, where
$D$ is diagonal. If $p(z,w)$ denotes the characteristic polynomial of $(A,B)$, then
\begin{eqnarray*}
p(z,w)&=&\det(I+zA+wB)\\
&=&\det(U*(I+zD+wUBU^*)U)\\
&=&\det(I+zD+wUBU^*)\\
&=&q(z,w),
\end{eqnarray*}
where $q(z,w)$ is the characteristic polynomial for the pair $(D,UBU^*)$.

On the other hand,
\begin{eqnarray*}
AB=BA &\iff& U^*DUB=BU^*DU\\
&\iff& DUB=UBU^*DU\\
&\iff& D(UBU^*)=(UBU^*)D.
\end{eqnarray*}
So $A$ commutes with $B$ if and only if $D$ commutes with $UBU^*$.

It is also easy to verify that $B$ is normal if and only if $UBU^*$ is normal.
Therefore, we have reduced the problem for $2\times2$ matrices to the case when
$A$ is diagonal and $B$ is normal.

Thus we consider the case in which
$$A=\begin{pmatrix} d_1 & 0\cr 0 & d_2\end{pmatrix},\qquad
B=\begin{pmatrix} a & b \cr c & d\end{pmatrix}.$$
A direct calculation shows that $B$ is normal if and only if
\begin{equation}
|b|=|c|,\quad a\overline c+b\overline d=\overline ab+\overline cd.
\label{eq1}
\end{equation}
Another direct calculation shows that $AB=BA$ if and only if
$$d_1=d_2\quad{\rm or}\quad b=c=0.$$
Each of these two conditions implies that $A$ and $B$ are simutaneously
diagonalizable by the same unitary matrix. When $A$ and $B$ are diagonalizable by the
same unitary matrix, it is easy to see that the characteristic polynomial $p$ for
the pair $(A,B)$ is the product of two linear polynomials, and the joint point
spectrum $\sigma_p(A,B)$ is the union of two complex lines (it is possible for them
to degenerate to one) in $\C^2$.

To prove the other direction, we begin with
$$I+zA+wB=\begin{pmatrix} d_1z+aw+1 & bw\cr cw & d_2z+dw+1\end{pmatrix}.$$
The characteristic polynomial $p$ for the pair $(A,B)$ is given by
$$p(z,w)=d_1d_2z^2+(ad-bc)w^2+(ad_2+dd_1)zw+(d_1+d_2)z+(a+d)w+1.$$
We want to see when the polynomial $p(z,w)$ is completely reducible to
linear polynomials. In particular, we want to show that if $p(z,w)$ is
completely reducible, then $A$ and $B$ commute.

By comparing coefficients, we see that
$$p(z,w)=(\lambda_1z+\mu_1w+1)(\lambda_2z+\mu_2w+1)$$
if and only if
\begin{equation}
\begin{cases}
\lambda_1\lambda_2&=d_1d_2\cr
\lambda_1+\lambda_2&=d_1+d_2\cr
\mu_1\mu_2&=ad-bc\cr
\mu_1+\mu_2&=a+d\cr
\lambda_1\mu_2+\mu_1\lambda_2&=ad_2+dd_1.
\end{cases}
\label{eq2}
\end{equation}

From the first two conditions in (\ref{eq2}) we can solve for $\lambda_k$ to obtain
$$\lambda_1=d_1,\lambda_2=d_2,\quad{\rm or}\quad \lambda_1=d_2,\lambda_2=d_1.$$
From the next two conditions in (\ref{eq2}) we can solve for $\mu_k$ to obtain
$$\mu_1=\frac{a+d\pm\sqrt{(a-d)^2+4bc}}{2},\quad \mu_2=\frac{a+d\mp\sqrt{(a-d)^2+4bc}}{2}.$$

Choosing $\lambda_1=d_1$, $\lambda_2=d_2$, $\mu_1$ with the plus sign, and $\mu_2$ with
the minus sign, we obtain
$$\lambda_1\mu_2+\mu_1\lambda_2=\frac12(d_1+d_2)(a+d)+\frac12(d_2-d_1)\sqrt{(a-d)^2+4bc}.$$
So the fifth condition in (\ref{eq2}), which we call the compatibility condition, becomes
\begin{eqnarray*}
(d_2-d_1)\sqrt{(a-d)^2+4bc}&=&2(ad_2+dd_1)-(d_1+d_2)(a+d)\\
&=&(d_2-d_1)(a-d).
\end{eqnarray*}

Now suppose the polynomial $p(z,w)$ is completely reducible, so that
\begin{equation}
(d_2-d_1)\sqrt{(a-d)^2+4bc}=(d_2-d_1)(a-d).
\label{eq3}
\end{equation}
There are two cases to consider. If $d_1=d_2$, then $A$ is a multiple of the identity matrix,
so $A$ commutes with $B$. If $d_1\not=d_2$, then
$$\sqrt{(a-d)^2+4bc}=a-d.$$
Squaring both sides gives us $bc=0$. Combining this with (\ref{eq1}), we obtain $b=c=0$, so
that $B$ is diagonal and commutes with $A$.

The three remaining choices for $\{\lambda_1,\lambda_2,\mu_1,\mu_2\}$ are handled in
exactly the same way. This completes the proof of our main result for $2\times2$
normal matrices.

It is of course possible that other (potentially simpler) approaches exist for the case of
$2\times2$ matrices. It is however difficult for us to imagine that a computational approach
can be found that would work for $N\times N$ matrices in general.

\section{The projective spectrum}

Recall that for for an operator tuple $\bA=(A_1,\cdots,A_n)$ on a Hilbert space $H$
the projective spectrum is the set $\Sigma(\bA)$ consisting of points $z=(z_1,\cdots,
z_n)\in\cn$ such that the operator $z_1A_1+\cdots+z_nA_n$ is not invertible. We will
also consider the set $\Sigma_p(\bA)$ of points $(z_1,\cdots,z_n)$ in $\cn$ such that
the operator $z_1A_1+\cdots+z_nA_n$ has a nontrivial kernel. This set will be called
the projective point spectrum of $\bA$.

From the Introduction and from the classical definition of spectrum for a
single operator we see that it is often necessary to append the identity operator $I$ to
any operator tuple we wish to study. In particular, we show that for any compact operator
tuple $(A_1,\cdots,A_n)$ the projective spectrum and the projective point spectrum for
the expanded tuple $(A_1,\cdots,A_n, I)$ are essentially the same.

\begin{prop}
Suppose $(A_1,\cdots,A_n)$ is a tuple of compact operators on $H$ and
$\bA=(A_1,\cdots,A_n,I)$. Then
$$\Sigma(\bA)\setminus\{z_{n+1}=0\}=\Sigma_p(\bA)\setminus\{z_{n+1}=0\}.$$
\label{1}
\end{prop}

\begin{proof}
It is obvious that the projective point spectrum is contained in the projective spectrum.
Now suppose $z_{n+1}\not=0$ and
$$(z_1,\cdots,z_n,z_{n+1})\in\Sigma(\bA).$$
Then the operator
$$T=z_1A_1+\cdots+z_nA_n+z_{n+1}I$$
is not invertible. We wish to show that $T$ has a nontrivial kernel.

By Atkinson's theorem (see \cite{D} for example), the operator $T$ is Fredholm and
has Fredholm index $0$, because its image in the Calkin algebra is $z_{n+1}$ times the
identity. Therefore, $T$ has closed range, and its kernel and cokernel have the same
finite dimension. Since $T$ is not invertible, we conclude that $T$ has a nontrivial,
finite-dimensional kernel.
\end{proof}

When the identity operator is included in the operator tuple
$$\bA=(A_1,\cdots,A_n,I),$$
we often need to consider the sets
$$\Sigma(\bA)\setminus\{z_{n+1}=0\},\quad\Sigma_p(\bA)\setminus\{z_{n+1}=0\}.$$
It is thus more convenient for us to modify the definition of the projective
spectrum and the projective point spectrum in such situations. Recall from the
Introduction that for an $n$-tuple $\bA=(A_1,\cdots,A_n)$ (not necessarily containing
the identity operator) we define $\sigma(\bA)$ to be the set of points
$z=(z_1,\cdots,z_n)\in\cn$ such that the operator
$$A=z_1A_1+\cdots+z_nA_n+I$$
is not invertible. Similarly, we define $\sigma_p(\bA)$ to be the set of points
$z=(z_1,\cdots,z_n)\in\cn$ such that the operator $A$ above has a nontrivial kernel.
The sets $\sigma(\bA)$ and $\sigma_p(\bA)$ are no longer ``projective'' and should be
considered as subsets of $\cn$ instead.

It is clear that if $z_{n+1}\not=0$, then
$$(z_1,\cdots,z_n,z_{n+1})\in\Sigma(A_1,\cdots,A_n,I)$$
if and only if
$$\left(\frac{z_1}{z_{n+1}},\cdots,\frac{z_n}{z_{n+1}}\right)\in\sigma(A_1,\cdots,A_n).$$
Similarly, if $z_{n+1}\not=0$, then
$$(z_1,\cdots,z_n,z_{n+1})\in\Sigma_p(A_1,\cdots,A_n,I)$$
if and only if
$$\left(\frac{z_1}{z_{n+1}},\cdots,\frac{z_n}{z_{n+1}}\right)\in\sigma_p(A_1,\cdots,A_n).$$
Therefore, any condition in terms of $\sigma(\bA)$ or $\sigma_p(\bA)$ can be rephrased
in terms of the projective spectrum and the projective point spectrum of
$$\bA'=(A_1,\cdots,A_n,I)$$
away from $z_{n+1}=0$, and vise versa. In particular, the following result is a
consequence of Proposition~\ref{1}.

\begin{prop}
If $\bA=(A_1,\cdots,A_n)$ is an $n$-tuple of compact operators on a Hilbert space $H$, then
$\sigma(\bA)=\sigma_p(\bA)$.
\label{2}
\end{prop}

Our main focus in the paper is on the relationship between the geometry of the projective
spectrum and the mutual commutativity of an operator tuple. The following result shows that
any linear structure in the projective spectrum is preserved under linear changes of variables.

\begin{lem}
Suppose $\bA=(A_1,\cdots,A_n)$ is an $n$-tuple of bounded linear operators on a Hilbert
space $H$ and
$$C=\begin{pmatrix}c_{11} & c_{12} & \cdots & c_{1n}\cr
                   c_{21} & c_{22} & \cdots & c_{2n}\cr
                   \cdots & \cdots & \cdots & \cdots\cr
                   c_{n1} & c_{n2} & \cdots & c_{nn}\end{pmatrix}$$
is an invertible matrix of complex numbers. If $\bB=(B_1,\cdots,B_n)$, where
$$\begin{cases}
B_1=c_{11}A_1+c_{12}A_2+\cdots+c_{1n}A_n\cr
B_2=c_{21}A_1+c_{22}A_2+\cdots+c_{2n}A_n\cr
\vdots\cr
B_n=c_{n1}A_1+c_{n2}A_2+\cdots+c_{nn}A_n,
\end{cases}$$
then $z=(z_1,\cdots,z_n)\in\sigma(\bB)$ if and only if $w=(w_1,\cdots,w_n)\in\sigma(\bA)$,
where $w=zC$ as matrix multiplication. Furthermore, the complex hyperplane
$$a_1z_1+\cdots+a_nz_n+1=0$$
is contained in $\sigma(\bB)$ if and only if the complex hyperplane
$$b_1z_1+\cdots+b_nz_n+1=0$$
is contained in $\sigma(\bA)$, where
$$(a_1,\cdots,a_n)=(b_1,\cdots,b_n)C^T$$
as matrix multiplication. The same results hold if $\sigma(\bA)$ and $\sigma(\bB)$ are
replaced by $\sigma_p(\bA)$ and $\sigma_p(\bB)$, respectively.
\label{3}
\end{lem}

\begin{proof}
Formally, we can write
\begin{eqnarray*}
z_1B_1+\cdots+z_nB_n+I&=&(z_1,\cdots,z_n)(B_1,\cdots,B_n)^T+I\\
&=&(z_1,\cdots,z_n)C(A_1,\cdots,A_n)^T+I\\
&=&(w_1,\cdots,w_n)(A_1,\cdots,A_n)^T+I\\
&=&w_1A_1+\cdots+w_nA_n+I
\end{eqnarray*}
as matrix multiplication. This immediately gives the relationship between $\sigma(\bA)$
and $\sigma(\bB)$.

Similarly, a point $z=(z_1,\cdots,z_n)\in\sigma(\bB)$ satisfies the condition
$$a_1z_1+\cdots+a_nz_n+1=0$$
if and only if
$$(z_1,\cdots,z_n)(a_1,\cdots,a_n)^T+1=0$$
if and only if
$$(w_1,\cdots,w_n)C^{-1}(a_1,\cdots,a_n)^T+1=0$$
if and only if
$$b_1w_1+\cdots+b_nw_n+1=0.$$
This completes the proof of the lemma.
\end{proof}

The next few lemmas discuss the case in which a complex hyperplane is
contained in $\sigma_p(\bA)$. These results, some of which are very technical,
contain the main new ideas of the paper and represent the major steps in the
proof of our main results.

\begin{lem}
Suppose the complex hyperplane
\begin{equation}
\lambda_1z_1+\cdots+\lambda_nz_n+1=0
\label{eq4}
\end{equation}
is contained in $\sigma_p(A_1,\cdots,A_n)$. Then each nonzero $\lambda_k$
is an eigenvalue of $A_k$. If $\lambda_k=0$ and $A_k$ has closed range, then
$0$ is an eigenvalue of $A_k$.
\label{4}
\end{lem}

\begin{proof}
It suffices to consider the case $k=1$.

First assume that $\lambda_1\not=0$. Since the point
$$\left(-\frac1{\lambda_1},0,\cdots,0\right)$$
belongs to the complex hyperplane in (\ref{eq4}), the operator
$$-\frac1{\lambda_1}A_1+I=\frac1{\lambda_1}(\lambda_1I-A_1)$$
has a nontrivial kernel, which means that $\lambda_1$ is an eigenvalue of $A_1$.

Next assume that $\lambda_1=0$. It is clear that at least one of the other $\lambda_k$'s
must be nonzero. Without loss of generality, let us assume that $\lambda_2\not=0$.
Consider the operator tuple $(B_1,B_2,\cdots,B_n)$, where $B_1=A_1+\varepsilon A_2$ and
$B_k=A_k$ for $2\le k\le n$. It follows from the previous paragraph and Lemma~\ref{3} that
$\lambda_\varepsilon:=\varepsilon\lambda_2$ is an eigenvalue of $A_\varepsilon:=A_1+
\varepsilon A_2$. For each $\varepsilon>0$ let $x_\varepsilon$ be a unit eigenvector of
$A_\varepsilon$. Since
\begin{eqnarray*}
\|A_1x_\varepsilon\|&\le&\|A_\varepsilon x_\varepsilon\|+(A_\varepsilon-A_1)x_\varepsilon\|\\
&\le&|\lambda_2|\varepsilon+\|A_\varepsilon-A_1\|\\&=&(|\lambda_2|+\|A_2\|)\varepsilon,
\end{eqnarray*}
it follows that the operator $A_1$ is not bounded below. If we also assume that $A_1$ has
closed range, then we can conclude that $A_1$ has a nontrivial kernel (otherwise, it follows
from the open mapping theorem that it must be bounded below). In other words, $0$ is an
eigenvalue of $A_1$.
\end{proof}

Note that if $\dim(H)<\infty$, then every operator on $H$ has closed range. In particular,
every $N\times N$ matrix, when considered as a linear operator on $\C^N$, has closed
range.

Also note that if $z$ and $w$ satisfy the equation $\lambda z+\mu w+1=0$, then the operator
$A(z)=zA+wB+I$ can be written as
\begin{eqnarray*}
A(z)&=&(\lambda z+1)\left(I-\frac1\mu\,B\right)-\lambda z\left(I-\frac1\lambda\,A\right)\\
&=&t\left(I-\frac1\mu\,B\right)+(1-t)\left(I-\frac1\lambda\,A\right),
\end{eqnarray*}
where $t=1+\lambda z\in\C$. This shows that if $x$ is a common eigenvector for $A$ and
$B$ corresponding to $\lambda$ and $\mu$, respectively, then $x$ belongs to the kernel
of each $A(z)$.

The next two lemmas will allow us to find common eigenvectors for $A$ and $B$ when
$\sigma_p(A,B)$ satisfies certain geometric conditions, for example, when $\sigma_p(A,B)$
contains a complex line in $\C^2$. If $A$ and $B$ are normal matrices of the same size,
then it is well known that they commute if and only if they can be diagonalized by the
same orthonormal basis. Therefore, the commutativity of $A$ and $B$ boils down to the
existence of sufficiently many common eigenvectors. The central idea of the paper is then
how to use certain geometric properties of $\sigma_p(A,B)$ to produce common eigenvectors
for $A$ and $B$.

\begin{lem}
Suppose $A$ and $B$ are both self-adjoint and compact. If the complex line
$\lambda z+\mu w+1=0$ is contained in $\sigma_p(A,B)$, where $\lambda\not=0$ and
$\mu\not=0$, then there exists a unit vector $x\in H$ such that $Ax=\lambda x$
and $\mu=\langle Bx,x\rangle$.
\label{5}
\end{lem}

\begin{proof}
It follows from Lemma~\ref{4} that $\lambda$ and $\mu$ are eigenvalues of $A$
and $B$, respectively.

Since both $A$ and $B$ are self-adjoint, the eigenvalues $\lambda$ and $\mu$ are real.
Let $\varepsilon$ be any small real number such that $\lambda+\mu\varepsilon\not=0$.
Choose $z$ and $w$ such that
$$-\frac1z=\lambda+\mu\varepsilon,\quad w=-\frac{1+\lambda z}\mu.$$
Then we have
$$\varepsilon=-\frac{1+\lambda z}{\mu z},\qquad \lambda z+\mu w+1=0.$$
It follows that
$$zA+wB+I=z\left(A-\frac{1+\lambda z}{\mu z}\,B+\frac1z\,I\right)
=z(A_\varepsilon-\lambda_\varepsilon I)$$
has nontrivial kernel, where
$$A_\varepsilon=A+\varepsilon B,\qquad \lambda_\varepsilon=\lambda+\mu\varepsilon.$$
In particular, $\lambda_\varepsilon$ is an eigenvalue of $A_\varepsilon$ and
\begin{equation}
\frac{d\lambda_\varepsilon}{d\varepsilon}(0)=\mu.
\label{eq5}
\end{equation}

Because $\varepsilon$ is real, the operator $A_\varepsilon$ is self-adjoint (and,
of course, compact). Since $\lambda$ is an isolated eigenvalue of $A$, there exists a
positive number $\delta$ such that $\lambda$ is the only eigenvalue of $A$ in the
Euclidean disk $D(\lambda,\delta)\subset\C$ and $uI-A$ is invertible on $|u-\lambda|=
\delta$. Let $N$ denote the multiplicity of $\lambda$ and $H_\lambda$ denote the
eigenspace of $A$ corresponding to $\lambda$. In particular, $\dim(H_\lambda)=N$.

There exists a positive number $\sigma$ such that for any real $\varepsilon$ with
$|\varepsilon<|\sigma$ the operator $uI-A_\varepsilon$ is invertible on $|u-\lambda|=
\delta$. For such $\varepsilon$ we consider the Riesz projections
\begin{equation}
P_\varepsilon=\frac1{2\pi i}\int_{\partial D(\lambda,\delta)}(uI-A_\varepsilon)^{-1}\,du,
\label{eq6}
\end{equation}
and
\begin{equation}
P_0=\frac1{2\pi i}\int_{\partial D(\lambda,\delta)}(uI-A)^{-1}\,du.
\label{eq7}
\end{equation}
Since each $A_\varepsilon$ is self-adjoint, the multiplicity of each eigenvalue of
$A_\varepsilon$ is equal to the dimension of the range of the corresponding Riesz projection.
By Theorem 3.1 on page 14 of \cite{GK}, the operator $A_\varepsilon$ has exactly $N$
eigenvalues in $D(\lambda, \delta)$, counting multiplicities, with $\lambda_\varepsilon$
being one of them. Let $E_\varepsilon$ denote the eigenspace of $A_\varepsilon$
corresponding to $\lambda_\varepsilon$.

It is well known that $P_0$ is the orthogonal projection from $H$ onto $H_\lambda$.
Since $A_\varepsilon$ is self-adjoint, $P_\varepsilon$ is an orthogonal projection too.
Suppose
$$\{\lambda_{\varepsilon,1},\cdots,\lambda_{\varepsilon,k}\}$$
are the distinct eigenvalues of $A_\varepsilon$ in $D(\lambda,\delta)$,
$$\{E_{\varepsilon,1},\cdots,E_{\varepsilon,k}\}$$
are the corresponding eigenspaces, and
$$\{P_{\varepsilon,1},\cdots,P_{\varepsilon,k}\}$$
are the associated orthogonal projections. Here $\lambda_{\varepsilon,1}=
\lambda_\varepsilon$. We then have
$$E_\varepsilon=E_{\varepsilon,1}\oplus\cdots\oplus E_{\varepsilon,k},\quad
P_\varepsilon=P_{\varepsilon,1}\oplus\cdots\oplus P_{\varepsilon,k}.$$

It is easy to check that $P_\varepsilon\to P_0$ as $\varepsilon\to0$ and we may assume that
$\sigma$ was chosen so that $\dim E_\varepsilon=\dim H_\lambda$ for all $|\varepsilon|
<\sigma$. It is also easy to check that  $\dim E_{\varepsilon,1}$ remains constant for
$\varepsilon$ small enough. Thus, we may also assume that $\dim E_{\varepsilon,1}=n_1$ for
$|\varepsilon|<\sigma$ and $P_{\varepsilon,1}\to P_1$ as $\varepsilon\to0$, where
$P_1$ is the orthogonal projection from $H$ onto a closed subspace $H_1\subset H_\lambda$
with $\dim H_1=n_1\le N$.

Fix any unit vector $v\in H_1$, so that $Av=\lambda v$. We are going to show that
$\mu=\langle Bv,v\rangle$. To this end, we consider the vector $v_\varepsilon=
P_\varepsilon v$ and use $\int$ to denote $\int_{|u-\lambda|=\delta}$. Then
\begin{eqnarray*}
v_\varepsilon&=&\frac1{2\pi i}\int(uI-A-\varepsilon B)^{-1}v\,du\\
&=&\frac1{2\pi i}\int\left(I-\varepsilon(uI-A)^{-1}B\right)^{-1}(uI-A)^{-1}v\,du\\
&=&\sum_{k=0}^\infty\frac{\varepsilon^k}{2\pi i}\int\left((uI-
A)^{-1}B\right)^k(uI-A)^{-1}v\,du\\
&=&\frac1{2\pi i}\int(uI-A)^{-1}v\,du\\
&&\ +\frac\varepsilon{2\pi i}\int(uI-A)^{-1}B(uI-A)^{-1}v\,du+O(\varepsilon^2)\\
&=&v+\frac\varepsilon{2\pi i}\int\frac{(uI-A)^{-1}Bv}{u-\lambda}\,du+O(\varepsilon^2)\\
&=&v+\varepsilon\tilde v+O(\varepsilon^2),
\end{eqnarray*}
where
$$\tilde v=\frac1{2\pi i}\int\frac{(uI-A)^{-1}Bv}{u-\lambda}\,du.$$

Since $A$ is compact and self-adjoint, there exists an orthonormal basis $\{v_k\}$ of
$H$ consisting of eigenvectors of $A$. If we write
$$Bv=\sum_{I}c_kv_k+\sum_{II}c_kv_k,$$
where $Av_k=\lambda v_k$ in the first sum above and $Av_k=\lambda_kv_k$ with $\lambda_k
\not=\lambda$ in the second sum above, then
$$\tilde v=\sum_{II}\frac{c_k}{\lambda-\lambda_k}\,v_k.$$
It follows that
$$\langle\tilde v,v\rangle=\sum_{II}\frac{c_k}{\lambda-\lambda_k}\langle v_k,v\rangle=0,$$
from which we deduce that
$$\|v_\varepsilon\|^2=1+O(\varepsilon^2).$$

On the other hand, we have
$$\langle A\tilde v,v\rangle=\sum_{II}\frac{c_k\lambda_k}{\lambda-\lambda_k}
\langle v_k,v\rangle=0.$$
Our choice of $v$ from $H_1$, which is the limit of $H_{\varepsilon,1}$,
ensures that $v_\varepsilon\in E_{\varepsilon,1}$ and so $A_\varepsilon v_\varepsilon=
\lambda_\varepsilon v_\varepsilon$. Therefore,
\begin{eqnarray*}
\lambda_\varepsilon&=&\frac{\langle A_\varepsilon v_\varepsilon,v_\varepsilon\rangle}
{\|v_\varepsilon\|^2}=\langle(A+\varepsilon B)v_\varepsilon,v_\varepsilon\rangle
+O(\varepsilon^2)\\
&=&\langle Av_\varepsilon,v_\varepsilon\rangle+\varepsilon\langle Bv_\varepsilon,
v_\varepsilon\rangle+O(\varepsilon^2)\\
&=&\langle Av+\varepsilon A\tilde v+O(\varepsilon^2),v+\varepsilon\tilde v+O(\varepsilon^2)\\
&&\ +\,\varepsilon\langle Bv+\varepsilon B\tilde v+O(\varepsilon^2),v+\varepsilon\tilde v
+O(\varepsilon^2)\rangle+O(\varepsilon^2)\\
&=&\lambda+\varepsilon\langle Bv,v\rangle+O(\varepsilon^2).
\end{eqnarray*}
This implies that
\begin{equation}
\frac{d\lambda_{\varepsilon}}{d\varepsilon}(0)=\langle Bv,v\rangle.
\label{eq8}
\end{equation}
Combining (\ref{eq5}) with (\ref{eq8}), we conclude that $\mu=\langle Bv,v\rangle$ and
the proof is complete.
\end{proof}

Note that Lemma~\ref{5} still holds if we only assumed that the real line
$$\lambda x+\mu y+1=0$$
in $\R^2$ is contained in $\sigma_p(A,B)$. Also, the condition of $A$
and $B$ being self-adjoint in Lemma \ref{5} was imposed to guarantee that the perturbed Riesz
projection $P_\varepsilon$ takes an eigenvector of $A$ to an eigenvector of $A_\varepsilon$.
Our next result shows that if $\lambda$ is a simple eigenvalue of $A$, then we
just need the operator $A$ to be normal (no assumption on $B$ is necessary), and the
geometric condition on $\sigma_p(A,B)$ can be relaxed.

Recall from \cite{SYZ} that for any compact operators $A$ and $B$ the joint spectrum
$\sigma_p(A,B)$ is an analytic set of codimension $1$ in $\C^2$. In other words, for any
point $(z_0,w_0)\in\sigma_p(A,B)$, there exists a neighborhood $U$ of $(z_0,w_0)$ and a
holomorphic function $F(z,w)$ on $U$ such that
$$\sigma_p(A,B)\cap U=\{(z,w)\in U:F(z,w)=0\}.$$
Further recall that a point of an analytic set is called regular if near this point the set
is a complex manifold. If a point is not regular, it is called singular. Here we are dealing
with analytic sets of pure codimension one. It is well-known that in this case if the set
has multiplicity one (if we consider an analytic set as a divisor with multiplicities), a point
is singular if and only if the differential of the local defining function vanishes at this
point. It is also well-known that the singularity of a point is independent of the choice
of the defining function and that the set of singular points has higher codimension.
These facts together with more advanced results on analytic sets can be found
in \cite{Ch}. For reasons mentioned above we call a point $(z_0,w_0)\in\sigma_p(A,B)$
singular if
$$\frac{\partial F}{\partial z}(z_0,w_0)=\frac{\partial F}{\partial w}(z_0,w_0)=0,$$
where $F$ is a local defining function for $\sigma_p(A,B)$.

\begin{lem}
Suppose $A$ and $B$ are both compact, $A$ is normal, $\lambda$ is a nonzero eigenvalue of
$A$ of multiplicity one, and $(-1/\lambda,0)$ is not a singular point of the analytic set
$\sigma_p(A,B)$. Then there exists a unit vector $x\in H$ such that $Ax=\lambda x$ and
$\mu=\langle Bx,x\rangle$, where $\lambda z+\mu w+1=0$ is tangent to $\sigma_p(A,B)$
at $(-1/\lambda,0)$.
\label{6}
\end{lem}

\begin{proof}
For $\varepsilon$ close to $0$ we write $A_\varepsilon=A+\varepsilon B$. By the continuity
of spectrum (see \cite{CM} for example), there exists an eigenvalue $\lambda_\varepsilon$
of $A_\varepsilon$ that is close to $\lambda$. Since $\lambda\not=0$, we may as well
assume that $\lambda_\varepsilon\not=0$ for all small $\varepsilon$. It is then clear that
$(-1/\lambda_\varepsilon,-\varepsilon/\lambda_\varepsilon)\in\sigma_p(A,B)$. Therefore,
$(-1/\lambda_\varepsilon,-\varepsilon/\lambda_\varepsilon)$ is the intersection of the
analytic set $\sigma_p(A,B)$ and the complex line $w=\varepsilon z$ in $\C^2$.

Just as in the proof of Lemma \ref{5}, since $\lambda$ is an isolated eigenvalue of $A$,
there exists a positive $\delta$ such that $\lambda$ is the only eigenvalue of $A$ in
the Euclidean disk $D(\lambda,\delta)\subset\C$. Again we can also assume that $\sigma<|\lambda|$
and the operators $uI-A$ are invertible on $|u-\lambda|=\delta$ and conclude that there
exists a positive number $\sigma$ such that for any $\varepsilon$ with $|\varepsilon|<\sigma$
the operators $uI-A_\varepsilon$ are invertible on $|u-\lambda|=\delta$ and
$\lambda_\varepsilon$ is the only eigenvalue of $A_\varepsilon$ in $D(\lambda,\delta)$.

We now consider the corresponding Riesz projections in (\ref{eq6}) and (\ref{eq7}).

Since $A$ is normal, its associated Riesz projection $P_0$ is the orthogonal projection
onto the one-dimensional eigenspace of $A$ corresponding to the eigenvalue $\lambda$ (recall
that the multiplicity of $\lambda$ is one).

We are not making any assumptions about the compact operator $B$. So the operators
$A_\varepsilon$ are not necessarily normal, and the projections $P_\varepsilon$
are not necessarily orthogonal. However, $\|A_\varepsilon-A\|\to0$ easily implies that
$\|P_\varepsilon-P_0\|\to0$. So by Lemma 3.1 on page 13 of \cite{GK}, we may as well conclude
each $P_\varepsilon$ is a one-dimensional projection. Furthermore, since for a compact operator
the range of the Riesz projection contains all eigenspaces corresponding to the eigenvalues
inside the contour of integration (see \cite{GK}), the range of $P_\varepsilon$
being one-dimensional is the eigenspace of $A_\varepsilon$ corresponding to
$\lambda_\varepsilon$. It follows that
\begin{equation}
A_\varepsilon P_\varepsilon x=\lambda_\varepsilon P_\varepsilon x
\label{eq9}
\end{equation}
for every $x\in H$.

A computation similar to the one in Lemma \ref{5} now shows that
\begin{equation}
P_\varepsilon=P_0+\frac\varepsilon{2\pi i}\int (uI-A)^{-1}B(uI-A)^{-1}\,du
+O(\varepsilon^2).
\label{eq10}
\end{equation}
In particular, if $Av_1=\lambda v_1$, then
\begin{equation}
P_\varepsilon v_1=v_1+\frac\varepsilon{2\pi i}\int\frac{(uI-A)^{-1}Bv_1}{u-\lambda}
\,du+O(\varepsilon^2).
\label{eq11}
\end{equation}
Suppose that in a neighborhood of the point $(-1/\lambda,0)$ the analytic set
$\sigma_p(A,B)$ is given by the equation $F(z,w)=0$, where $F$ is holomorphic and
at least one of the two partial derivatives $\partial F/\partial z$ and
$\partial F/\partial w$ is nonvanishing in a (possibly small) neighborhood of
$(-1/\lambda,0)$.

Let us first consider the case in which $\partial F/\partial z$ is nonvanishing in a
neighborhood of $(-1/\lambda,0)$. By the implicit function theorem, there exists an
analytic function $z=\varphi(w)$, $|w|<r_0$, such that $\sigma_p(A,B)$ is
the analytic curve $z=\varphi(w)$ near the point $(-1/\lambda,0)$. In particular,
the equation (point-slope form) of the tangent line of $\sigma_p(A,B)$ at the point
$(-1/\lambda,0)$ is given by
$$z+\frac1\lambda=\varphi'(0)(w-0),$$
or
$$\lambda z-\varphi'(0)w+1=0.$$
Consequently, $\mu=-\lambda\varphi'(0)$.

Since $\varphi(0)\not=0$, we may assume that $r_0$ is small enough so that the function
$\psi(w)=w/\varphi(w)$ is well defined and analytic for $|w|<r_0$. Let $\varepsilon=
\psi(w)$ for $|w|<r_0$. Then $\psi(0)=0$, $\psi'(0)=-\lambda$, and
\begin{equation}
\psi(w)=-\lambda w+O(|w|^2).
\label{eq12}
\end{equation}
Since $(\varphi(w),w)\in\sigma_p(A,B)$ for $|w|<r_0$, the operators $\varphi(w)A+wB+I$
have nontrivial kernels for $|w|<r_0$. Equivalently, the operators
$$A+\frac w{\varphi(w)}B+\frac1{\varphi(w)}I=A+\varepsilon B+\frac1{\varphi(w)}I$$
have nontrivial kernels for $|w|<r_0$. This shows that
\begin{equation}
\lambda_\varepsilon=-\frac1{\varphi(w)}=\lambda+\lambda^2\varphi'(0)w+O(|w|^2)
\label{eq13}
\end{equation}
for $|w|$ (or equivalently $\varepsilon$) sufficiently small.

By (\ref{eq13}), we have
\begin{equation}
\varepsilon=\frac w{\varphi(w)}=-w(\lambda+\lambda^2\varphi'(0)w+O(|w|^2))=
-\lambda w+O(|w|^2).
\label{eq14}
\end{equation}
Therefore, we can rewrite (\ref{eq10}) in terms of $w$ as follows:
\begin{equation}
P_w=P_0-\frac{\lambda w}{2\pi i}\int(uI-A)^{-1}B(uI-A)^{-1}\,du+O(|w|^2).
\label{eq15}
\end{equation}
Let $\{v_1,v_2,v_3,\cdots,\}$ be an orthonormal basis of $H$ consisting of eigenvectors
of $A$ with $Av_1=\lambda v_1$. Since $\lambda$ is an eigenvalue of multiplicity $1$, we
have $Av_k=\lambda_k v_k$, $\lambda_k\not=\lambda$, for all $k\ge2$.

By equations (\ref{eq9}) and (\ref{eq13})-(\ref{eq15}), we have
\begin{eqnarray*}
(A-\lambda wB+O(|w|^2))P_\varepsilon v_1&=&(A_\varepsilon+O(|w|^2))P_\varepsilon v_1\\
&=&A_\varepsilon P_\varepsilon v_1+O(|w|^2)\\
&=&\lambda_\varepsilon P_\varepsilon v_1+O(|w|^2)\\
&=&(\lambda+\lambda^2\varphi'(0)w+O(|w|^2))P_\varepsilon v_1.
\end{eqnarray*}
This along with (\ref{eq11}) shows that
$$\left(A-\lambda wB+O(|w|^2)\right)\left[v_1-\frac{\lambda w}{2\pi i}\int
\frac{(uI-A)^{-1}Bv_1}{u-\lambda}\,du+O(|w|^2)\right]$$
is equal to
$$(\lambda+\lambda^2\varphi'(0)w+O(|w|^2))\left[v_1-\frac{\lambda w}{2\pi i}
\int\frac{(uI-A)^{-1}Bv_1}{u-\lambda}\,du+O(|w|^2)\right].$$
Since
$$\frac1{2\pi i}\int\frac{(uI-A)^{-1}v_1}{u-\lambda}\,du=\frac1{2\pi i}\int
\frac{v_1du}{(u-\lambda)^2}=0,$$
and
$$\frac1{2\pi i}\int\frac{(uI-A)^{-1}v_j}{u-\lambda}\,du=\frac1{2\pi i}\int
\frac{v_j\,du}{(u-\lambda)(u-\lambda_j)}=\frac{v_j}{\lambda-\lambda_j}$$
for $j\ge2$, we have
\begin{eqnarray*}
\frac1{2\pi i}\int\frac{(uI-A)^{-1}Bv_1}{u-\lambda}\,du&=&\sum_{j=2}^\infty
\frac1{2\pi i}\int\frac{(uI-A)^{-1}\langle Bv_1,v_j\rangle}{u-\lambda}\,v_j\,du\\
&&\ +\frac1{2\pi i}\int\frac{(uI-A)^{-1}\langle Bv_1,v_1\rangle}{u-\lambda}\,v_1\,du\\
&=&\sum_{j=2}^\infty\frac{\langle Bv_1,v_j\rangle}{\lambda-\lambda_j}\,v_j.
\end{eqnarray*}
Thus
$$(A-\lambda wB+O(|w|^2))\left[v_1-\lambda w\sum_{j=2}^\infty\frac{\langle
Bv_1,v_j\rangle}{\lambda-\lambda_j}\,v_j+O(|w|^2)\right]$$
is equal to
$$(\lambda+\lambda^2\varphi'(0)w+O(|w|^2))\left[v_1-\lambda w
\sum_{j=2}^\infty\frac{\langle Bv_1,v_j\rangle}{\lambda-\lambda_j}\,v_j+O(|w|^2)\right].$$
Multiplying everything out, we obtain
$$-\lambda w\sum_{j=2}^\infty\frac{\langle Bv_1,v_j\rangle}{\lambda-\lambda_j}
\,Av_j-\lambda wBv_1$$
$$=\lambda^2\varphi^\prime(0)wv_1-\lambda^2w\sum_{j=2}^\infty\frac{\langle Bv_1,v_j\rangle}
{\lambda-\lambda_j}\,v_j+O(|w|^2).$$
Since $\langle Av_j,v_1\rangle=\lambda_j\langle v_j,v_1\rangle=0$ for $j\ge2$, taking
the inner product of both sides above with $v_1$ gives
$$-\lambda w\langle Bv_1,v_1\rangle=\lambda^2\varphi'(0)w+O(|w|^2),$$
which clearly gives
$$\langle Bv_1,v_1\rangle=-\lambda\varphi'(0)=\mu.$$
This completes the proof of the lemma in the case when $\partial F/\partial z$ is nonzero
at $(-1/\lambda,0)$. The case when
$$\frac{\partial F}{\partial z}\left(-\frac1\lambda,0\right)=0,\quad
\frac{\partial F}{\partial w}\left(-\frac1\lambda,0\right)\not=0,$$
is similar. But we will only need the case proved above.
\end{proof}

A careful examination of the proofs of the last two lemmas shows that, in Lemma~\ref{5},
the condition that the whole line $\lambda z+\mu w+1=0$ is contained in $\sigma_p(A,B)$
can be weakened to the form of Lemma~\ref{6}. Although we do not need this general result
for the proof of our main theorems, we think it is of some independent interest and will
state it as follows. The proof goes along the same lines as of Lemmas \ref{5} and \ref{6}.

\begin{lem}
Suppose $A$ and $B$ are both self-adjoint and compact. If $\sigma_p(A,B)$ contains a smooth
curve $\Gamma \subset \R^2\subset \C^2$ given by $\Gamma=\{ F(x,y)=0\}$ which passes through
the point $(-\frac{1}{\lambda},0)$, where $\lambda\not=0$, such that at least one of the
partial derivatives $\partial F/\partial x$ and $\partial F/\partial y$ does not vanish at
$(-\frac{1}{\lambda},0)$, and if the real line $\lambda x+\mu y+1=0$ is tangent to $\Gamma$ at
$(-\frac{1}{\lambda},0)$, then there exists a unit vector $v\in H$ such that $Av=\lambda v$
and $\mu=\langle Bv,v\rangle$.
\label{7}
\end{lem}

Recall that a holomorphic curve in $\C^2$ is a nonconstant holomorphic function
$F(u)=(f(u),g(u))$ from $\C$ into $\C^2$. We will denote this curve simply by $F$. The
following result shows that if ``sufficiently many'' points on the curve $F$ belongs
to $\sigma_p(A,B)$, then the whole curve is in $\sigma_p(A,B)$.

\begin{prop}
Let $F(u)=(f(u),g(u))$ be a holomorphic curve in $\C^2$. If there exists a sequence
$\{u_k\}\subset\C$, having at least one accummulation point in $\C$, such that the
points $\{F(u_k)\}$ all belong to $\sigma_p(A,B)$, where $A$ and $B$ are compact
operators, then the whole holomorphic curve $F$ is contained in $\sigma_p(A,B)$.
\label{8}
\end{prop}

\begin{proof}
We consider the holomorphic, operator-valued function
$$T(u)=-[f(u)A+g(u)B],\qquad u\in\C.$$
The point $F(u)=(f(u),g(u))$ belongs to $\sigma_p(A,B)$ if and only if the solution
space of $x-T(u)x=0$, $x\in H$, is nontrivial. The desired result then follows from
Theorem 5.1 on page 21 of \cite{GK}.
\end{proof}

As a consequence of the proposition above, we see that if a nontrivial segment of the
complex line $\lambda z+\mu w+1=0$ is contained in $\sigma_p(A,B)$, then the entire
line is contained in $\sigma_p(A,B)$.

\section{Compact self-adjoint operators}

In this section we consider the case of two compact and self-adjoint operators. In this case
the main result we obtain is easy to state and the proof is easy to understand. Recall from
Proposition~\ref{2} that when $A$ and $B$ are both compact we have $\sigma(A,B)=\sigma_p(A,B)$.

The next lemma shows that the assumption $\lambda\not=0$ in Lemma~\ref{5} can be removed,
provided that $\mu$ is an eigenvalue of $B$ with maximum modulus. This is the key to
our main results.

\begin{lem}
Suppose $A$ and $B$ are both compact and self-adjoint. If the complex line
$\lambda z+\mu w+1=0$ is contained in $\sigma_p(A,B)$, where $|\mu|=\|B\|>0$, then
$\lambda$ is an eigenvalue of $A$, $\mu$ is an eigenvalue of $B$, and they share an
eigenvector.
\label{9}
\end{lem}

\begin{proof}
Note that if $|\mu|=\|B\|>0$, then the condition $\mu=\langle Bx,x\rangle$ with
$\|x\|=1$ is equivalent to $Bx=\mu x$. This follows easily from the Cauchy Schwarz
inequality and the fact that equality holds in the Cauchy-Schwarz inequality if
and only if the two vectors are linearly dependent.

The case $\lambda\not=0$ follows from Lemma~\ref{5}.

Suppose $\lambda=0$ and the complex line $\lambda z+\mu w+1=0$ is contained
in $\sigma_p(A,B)$. Then $(z,-1/\mu)\in\sigma(A,B)$ for every $z\in\C$. In other
words, the operator
$$I+zA-\frac1\mu\,B$$
has a nontrivial kernel for every $z\in\C$. If $(z,w)$ satisfies $\mu z+\mu w+1=0$,
then $w=-(1+\mu z)/\mu$ and
\begin{eqnarray*}
I+z(A+B)+wB&=&I+z(A+B)-\frac{1+\mu z}\mu\,B\\
&=&I+zA-\frac1\mu\,B,
\end{eqnarray*}
which has a nontrivial kernel. This shows that the complex line
$$\mu z+\mu w+1=0$$
is contained in $\sigma_p(A+B,B)$. By Lemma \ref{5}, there exists a nonzero
vector $x\in H$ such that
$$\mu=\langle Bx,x\rangle,\quad (A+B)x=\mu x.$$
The assumption $|\mu|=\|B\|$ along with $\mu=\langle Bx,x\rangle$ implies that
$Bx=\mu x$ and so $Ax=0$. This shows that $\lambda=0$ is an eigenvalue of $A$,
and the eigenvector $x$ is shared by $A$ and $B$.
\end{proof}

The following result is well known, but we include a proof here for
the sake of completeness.

\begin{lem}
Suppose $A$ and $B$ are both compact and normal on $H$. Then $AB=BA$ if and only if
they can be diagonalized under the same orthonormal basis.
\label{10}
\end{lem}

\begin{proof}
If $A$ and $B$ are simultaneously diagonalizable by the same unitary operator,
it is obvious that $A$ and $B$ will commute.

To prove the other direction, we write $A=\sum_{k=0}^\infty\lambda_kP_k$, where
$\{\lambda_k\}$ is the sequence of distinct eigenvalues of $A$ and $\{P_k\}$ is
the sequence of spectral projections (orthogonal projections onto the corresponding
eigenspaces $E_k$); see \cite{Z}. It is well known (see \cite{DS} for example) from
the spectral theory for normal operators that $AB=BA$ if and only if $P_kB=BP_k$
for every $k$, or equivalently, every $E_k$ is a reducing subspace for $B$. So if
$A$ and $B$ commute, then under the same direct decomposition
$$H=E_0\oplus E_1\oplus E_2\oplus\cdots,$$
we have
$$A=\lambda_0I_0\oplus\lambda_1I_1\oplus\lambda_2I_2\oplus\cdots,$$
and
$$B=B_0\oplus B_1\oplus B_2\oplus\cdots,$$
where each $I_k$ is the identity operator on $E_k$ and each $B_k$ is normal on $E_k$.
Now for each $k$ choose a unitary operator $U_k$ to diagonalize $B_k$. Then the unitary
operator
$$U=U_0\oplus U_0\oplus U_1\oplus U_2\oplus\cdots$$
will diagonalize $A$ and $B$ simultaneously.
\end{proof}

We can now prove the main result of this section.

\begin{thm}
Suppose $A$ and $B$ are both compact and self-adjoint. Then $AB=BA$ if and only if
$\sigma_p(A,B)$ consists of countably many, locally finite, complex lines
$\lambda_kz+\mu_kw+1=0$.
\label{11}
\end{thm}

\begin{proof}
First assume that $AB=BA$. By Lemma~\ref{10}, there exists an orthonormal basis
$\{e_n\}$ of $H$ which simultaneously diagonalizes $A$ and $B$, say
$$A=\sum_{n=1}^\infty\lambda_ne_n\otimes e_n,\quad
B=\sum_{n=1}^\infty\mu_ne_n\otimes e_n.$$
It follows that
$$I+zA+wB=\sum_{n=1}^\infty(1+\lambda_nz+\mu_nw)e_n\otimes e_n,$$
which is invertible if and only if $\lambda_nz+\mu_nw+1\not=0$ for every $n$. This
shows that
$$\sigma_p(A,B)=\bigcup_{n=1}^\infty\left\{(z,w):\lambda_nz+\mu_nw+1=0\right\}.$$
In other words, the joint point spectrum $\sigma_p(A,B)$ is the union of countably many
complex lines. It is easy to check that these complex lines are locally finite.

Next assume that $\sigma_p(A,B)$ consists of a countable number of complex lines which
are locally finite. We start with an eigenvalue of maximum modulus for $B$, say $\mu_1$
with $\|B\|=|\mu_1|$. The point $(0,-1/\mu_1)$ belongs to $\sigma(A,B)$, because the
operator
$$I+0A-\frac1{\mu_1}\,B=\frac1{\mu_1}(\mu_1I-B)$$
has a nontrivial kernel. Since $\sigma_p(A,B)$ consists of a bunch of complex lines, we
can find a complex line $\lambda z+\mu w+1=0$ that is contained in $\sigma_p(A,B)$ and
contains the point $(0,-1/\mu_1)$. It is then clear that $\mu=\mu_1$, so the complex
line $\lambda z+\mu_1w+1=0$ is contained in $\sigma_p(A,B)$.

By Lemma~\ref{9}, $\lambda$ is an eigenvalue of $A$. Furthermore, there exists a
nontrivial subspace $E$ of $\ker(\lambda I-A)$ such that $Bx=\mu_1x$ for $x\in E$.
Let $H=E\oplus H_1$ and
$$A=\lambda I\oplus A_1,\qquad B=\mu_1 I\oplus B_1,$$
be the corresponding decompositions.

Switch to the new pair $(A_1,B_1)$, whose joint point spectrum $\sigma_p(A_1,B_1)$
is contained in $\sigma_p(A,B)$. We claim that $\sigma_p(A_1,B_1)$ is still the union
of countably many, locally finite, complex lines. To see this, suppose that
a point $(z_0,w_0)\in \sigma_p(A_1,B_1)$ belongs to the complex line $\lambda z+\mu w+1=0$
which is contained in $\sigma_p(A,B)$ and to no other line in $\sigma_p(A,B)$. Because of
the local finiteness there is some $\delta>0$ such that
the intersection of $\sigma_p(A,B)$ with
$$D_\delta=\{(z,w)\in\C^2: |(z,w)-(z_0,w_0)|<\delta\}$$
is contained in the complex line $\lambda z+\mu w+1=0$. By spectral continuity, there is
some $\varepsilon >0$ such that for
$$|(z_1,w_1)-(z_0,w_0)|<\varepsilon$$
there is an eigenvalue $\tau$ of $z_1A+w_1B$ satisfying $|1+\tau|<\delta$. This implies that $$(\frac{z_1}{\tau},\frac{w_1}{\tau})\in\sigma_p(A_1,B_1)\subset \sigma_p(A,B).$$
Thus,
$$(\frac{z_1}{\tau},\frac{w_1}{\tau})\in \{ \lambda z + \mu w+1=0 \}.$$
In particular, this implies that
$$D_\delta \cap \{\lambda z+ \mu w+1 =0\} \subset \sigma_p(A_1,B_1).$$
By Proposition \ref{8}, the whole line $\lambda z+\mu w+1=0$ is contained in
$\sigma_p(A_1,B_1)$. Therefore, $\sigma_p(A_1,B_1)$ is still the union of countably many,
locally finite, complex lines. Now start with an eigenvalue of $A_1$ with maximum modulus
and repeat the above process to get a new pair $(A_2,B_2)$. Continuing this process in an
alternating way, we arrive at a sequence of decompositions
$$H=X_n\oplus Y_n,\quad A=T_n+A_n,\quad B=S_n+B_n,$$
where
$$T_nS_n=S_nT_n,\quad \|A_{n}\|\to0,\quad\|B_n\|\to0,$$
as $n\to\infty$. Let $n\to\infty$. The result is $AB=BA$.
\end{proof}

It is just a simple step to generalize the theorem above to the case of more than
two operators.

\begin{thm}
Suppose $\bA=\{A_1,A_2,\cdots,A_n\}$ is a tuple of compact and self-adjoint operators
on a Hilbert space $H$. Then $A_iA_j=A_jA_i$ for all $i$ and $j$ if and only if
$\sigma_p(\bA)$ is the union of countably many, locally finite, complex hyperplanes
$\lambda_{1k}z_1+\lambda_{2k}z_2+\cdots+\lambda_{nk}z_n+1=0$, $k\ge1$.
\label{12}
\end{thm}

\begin{proof}
If the operators in $\bA$ pairwise commute, then it follows from the proof of
Lemma~\ref{10} that these operators can be diagonalized simultaneously using the same
orthonormal basis $\{e_k\}$:
$$A_j=\sum_{k=1}^\infty\lambda_{jk} e_k\otimes e_k,\qquad 1\le j\le n.$$
It follows that
$$z_1A_1+z_2A_2+\cdots+z_nA_n+I=\sum_{k=1}^\infty(\lambda_{1k}z_1+\lambda_{2k}z_2+
\cdots+\lambda_{nk}z_n+1)e_k\otimes e_k,$$
which is invertible if and only if
$$\lambda_{1k}z_1+\lambda_{2k}z_2+\cdots+\lambda_{nk}z_n+1\not=0,\qquad k\ge1.$$
Therefore,
$$\sigma_p(\bA)=\bigcup_{k=1}^\infty\{(z_1,\cdots,z_n)\in\cn:
\lambda_{1k}z_1+\lambda_{2k}z_2+\cdots+\lambda_{nk}z_n+1=0\}.$$
It is easy to check that these complex hyperplanes are locally finite in $\cn$.

On the other hand, if $\sigma_p(\bA)$ consists of countably many, locally finite,
complex hyperplanes in $\cn$, then for any fixed $1\le i<j\le n$, the joint point
spectrum $\sigma_p(A_i,A_j)$, which is equal to
$$\{(z_i,z_j)\in\C^2:(z_1,z_2,\cdots,z_n)\in\sigma_p(\bA),
z_l=0, l\not=i, l\not=j\},$$
consists of countably many, locally finite, complex lines in $\C^2$. By
Theorem~\ref{11}, we have $A_iA_j=A_jA_i$.
\end{proof}

\section{Normal matrices}

The most important tool in the study of a single matrix $A$ is probably its characteristic
polynomial
$$p(\lambda)=\det(\lambda I-A),$$
where $I$ is the identity matrix. To study several matrices ${\mathbb A}=\{A_1,\cdots,A_n\}$
of the same size, it is thus natural to consider the following polynomial:
$$p_{\mathbb A}(z_1,\cdots,z_n)=\det(I+z_1A_1+\cdots+z_nA_n).$$
We still call $p_{\mathbb A}$ the characteristic polynomial of $\{A_1,\cdots,A_n\}$.

The classical characteristic polynomial of an $N\times N$ matrix $A$ is always a polynomial
of degree $N$. However, the degree of $p_{\mathbb A}$ is not necessarily $N$; it is
always less than or equal to $N$.

Recall that a matrix $A$ is normal if $AA^*=A^*A$. Here $A^*$ means the transpose of the
complex conjugate of $A$. It is well known that $A$ is normal if and only if $A$ is
diagonalizable by a unitary matrix. Two normal matrices are not necessarily diagonalizable
by the same unitary matrix, so their commutativity is an interesting and nontrivial problem.
In this section we characterize the commutativity of an $n$-tuple of normal matrices based
on their joint spectrum and on the reducibility of their characteristic polynomial.

Our first step is to show that for matrices two of the assumptions in Lemma~\ref{6} can be
dropped. Recall that an eigenvalue of an operator is called simple if its eigenspace is one
dimensional. The terms ``simple eigenvalue'' and ``eigenvalue of multiplicity one'' mean
the same thing.

\begin{lem}
Suppose $A$ and $B$ are $N\times N$ matrices. If $A$ is normal and $\lambda$ is a simple
nonzero eigenvalue of $A$, then $(-1/\lambda,0)$ is a regular point of the algebraic set
$\sigma_p(A,B)$.
\label{13}
\end{lem}

\begin{proof}
Recall that $\sigma_p(A,B)$ is the zero variety of the polynomial
$$f(z,w)=\det(zA+wB+I),\qquad (z,w)\in\C^2.$$
If $\lambda$ is a nonzero eigenvalue of $A$, then it is clear that the point $P=(-1/\lambda,0)$
belongs to $\sigma_p(A,B)$. Since $A$ is normal, the algebraic and geometric multiplicities
of $\lambda$ are the same. Thus the characteristic polynomial of $A$ admits the factorization
$$\det(uI-A)=(u-\lambda)g(u),\qquad g(\lambda)\not=0.$$
From this we easily deduce that
$$\frac{\partial f}{\partial z}(P)=\left.\frac{d}{dz}\det(zA+I)\right|_{z=-\frac1\lambda}
\not=0,$$
so $P$ is a regular point of $\sigma_p(A,B)$.
\end{proof}

\begin{lem}
Let $A$ and $B$ be $N\times N$ matrices. If $A$ is normal and the complex line
$\lambda z+\mu w+1=0$ is contained in $\sigma_p(A,B)$, then there exists a unit vector
$x\in\C^N$ such that $Ax=\lambda x$ and $\mu=\langle Bx,x\rangle$.
\label{14}
\end{lem}

\begin{proof}
By Lemma \ref{4} and the remark following it, $\lambda$ is an eigenvalue of $A$ and $\mu$
is an eigenvalue of $B$. Without loss of generality (see Section 2) we may assume that $A$
is a diagonal matrix with diagonal entries $\{\lambda_1,\lambda_2,\cdots,\lambda_N\}$, where
$\lambda_1=\cdots=\lambda_m=\lambda$ with $1\le m\le N$, and $\lambda_k\not=\lambda$ for
$m<k\le N$.

For any positive integer $j$ let $A_j$ denote the matrix obtained from the diagonal matrix
$A$ by modifying its diagonal entries as follows: if a number $d$ appears $L$ times in the
diagonal, then replace those $L$ occurrances of $d$ by
$$d+\frac1j,d+\frac2j,\cdots,d+\frac Lj,$$
respectively. This way, we find a sequence
$$\varepsilon_j=\{\varepsilon_{j1},\varepsilon_{j2},\cdots,\varepsilon_{jN}\},\quad
j=1,2,3,\cdots,$$
in the positive cone of $\C^N$ such that $\varepsilon_j\to0$ as $j\to\infty$ and, for all
$j$ sufficiently large, the numbers
$$\lambda_1+\varepsilon_{j1},\lambda_2+\varepsilon_{j2},\cdots,
\lambda_N+\varepsilon_{jN},$$
are distinct, nonzero, and constitute the diagonal entries of $A_j$.

Each matrix $A_j$ is normal and its eigenvalues are all simple and nonzero.
Since $\|A_j-A\|\to0$ as $j\to\infty$, the continuity of spectrum (see \cite{CM}
for example) shows that $\sigma_p(A_j,B)$ converges to $\sigma_p(A,B)$ uniformly on
compact subsets of $\C^2$. When viewed geometrically, the complex line $\lambda z+\mu w+1
=0$ is an irreducible component of $\sigma_p(A,B)$. Note that algebraically, the linear
polynomial $\lambda z+\mu w+1$ may appear multiple times in the factorization of
$\det(zA+wB+I)$. For each $j$ we can choose an irreducible component of $\sigma_p(A_j,B)$,
denoted $\Sigma_j$, in such a way that $\Sigma_j$ converges to the complex line $\lambda z+
\mu w+1=0$ in $\C^2$ uniformly on compacta. We know all the eigenvalues of $A_j$, so we
will assume that, for each $j\ge1$, the component $\Sigma_j$ passes through the point
$$P_j=\left(-\frac1{\lambda_{k_j}+\varepsilon_{jk_j}},0\right),\qquad 1\le k_j\le m.$$

Since the component $\Sigma_j$ of $\sigma_p(A_j,B)$ is the zero set of a polynomial
factor $f_j$ of $f(z,w)=\det(zA_j+wB+1)$, and the uniform convergence of $\{f_j\}$ on
compacta implies that its partial derivatives converge uniformly on compacta as well.
Also, by Lemma~\ref{13}, each point $P_j$ is a regular point on the component $\Sigma_j$.
Therefore, the tangent line of $\Sigma_j$ at $P_j$ converges to the line
$\lambda z+\mu w+1=0$. Also, $\lambda_{k_j}+\varepsilon_{jk_j}\to\lambda$ as $j\to\infty$.

The tangent line of $\Sigma_j$ at the point $P_j$ is given by
$$(z+\frac1{\lambda_{k_j}+\varepsilon_{jk_j}})\frac{\partial f}{\partial z}(P_j)
+\frac{\partial f}{\partial w}(P_j)=0.$$
If we write
$$\partial_zf=\frac{\partial f}{\partial z},\quad \partial_wf=\frac{\partial f}{\partial w},$$
then the tangent line of $\Sigma_j$ at $P_j$ becomes
$$(\lambda_{k_j}+\varepsilon_{jk_j})z+\mu_jw+1=0,$$
where
$$\mu_j=(\lambda_{k_j}+\varepsilon_{jk_j})\frac{\partial_wf(P_j)}{\partial_zf(P_j)}.$$
Since the tangent line of $\Sigma_j$ at $P_j$ converges to the complex line
$\lambda z+\mu w+1=0$ and $\lambda_{k_j}+\varepsilon_{jk_j}\to\lambda$ as $j\to\infty$, we
have $\mu_j\to\mu$ as $j\to\infty$.

Since $P_j$ is a regular point of $\sigma_p(A_j,B)$ and $\lambda_{k_j}+\varepsilon_{jk_j}$
is a simple, nonzero eigenvalue of $A_j$, it follows from Lemma~\ref{5} that there exists a
unit vector $x_j$ in $\C^N$ such that
\begin{equation}
A_jx_j=(\lambda_{k_j}+\varepsilon_{jk_j})x_j,\qquad \mu_j=\langle Bx_j,x_j\rangle.
\label{eq16}
\end{equation}
The unit sphere in $\C^N$ is compact, so we may as well assume that $\{x_j\}$ converges to
a unit vector in $\cn$ as $j\to\infty$. Letting $j\to\infty$ in (\ref{eq16}), we obtain
$Ax=\lambda x$ and $\mu=\langle Bx,x\rangle$.
\end{proof}

Recall that if $|\mu|=\|B\|$, $x$ is a unit vector in $\C^N$, and
$\mu=\langle Bx,x\rangle$, then $\mu$ is an eigenvalue of $B$ and $Bx=\mu x$.

We can now prove the main result of this section.

\begin{thm}
Suppose $\bA=(A_1,\cdots,A_n)$ is an $n$-tuple of $N\times N$ normal matrices
over the complex field. Then the following conditions are equivalent:
\begin{enumerate}
\item[(a)] $A_iA_j=A_jA_i$ for all $1\le i,j\le n$.
\item[(b)] $\sigma_p(\bA)$ is the union of finitely many complex hyperplanes in $\cn$.
\item[(c)] The characteristic polynomial of $\bA$ is completely reducible.
\end{enumerate}
\label{15}
\end{thm}

\begin{proof}
With Lemma~\ref{14} replacing Lemma~\ref{9}, the proof for the equivalence of (a) and (b)
is now the same as as the proof of Theorems~\ref{11} and \ref{12}.

If condition (a) holds, then by Lemma~\ref{10}, we may assume that each $A_k$ is diagonal
with diagonal entries $\{\lambda_{k1},\cdots,\lambda_{kN}\}$. It is then easy to see that
the characteristic polynomial of $\bA$ is given by
$$p(z_1,\cdots,z_n)=\prod_{k=1}^N(1+\lambda_{1k}z_1+\cdots+\lambda_{nk}z_n),$$
which is completely reducible. This shows that condition (a) implies (c).

Recall that a matrix is invertible if and only if its determinant is nonzero. If the
characteristic polynomial of $\bA$ is completely reducible, say
$$p(z_1,\cdots,z_n)=\prod_{k=1}^N(1+\lambda_{1k}z_1+\cdots+\lambda_{nk}z_n),$$
then its zero set consists of the complex hyperplanes
$$\lambda_{1k}z_1+\cdots+\lambda_{nk}z_n+1=0,\qquad 1\le k\le N.$$
This shows that condition (c) implies (b).
\end{proof}

\section{A normality test}

In this section we present a normality test for compact operators in terms of
the joint point spectrum.

\begin{thm}
A compact operator $A$ is normal if and only if the joint point spectrum
$\sigma_p(A,A^*)$ consists of countably many, locally finite, complex lines
in $\C^2$.
\label{16}
\end{thm}

\begin{proof}
Consider the compact operators
$$A_1=A+A^*,\qquad A_2=i(A-A^*),$$
which are clearly self-adjoint. It is easy to see that $A$ is normal if and only
if $A_1$ and $A_2$ commute. By Theorem~\ref{11}, $A_1$ and $A_2$ commute if and only
if the joint point spectrum $\sigma_p(A+A^*, i(A-A^*))$ consists of countably many,
locally finite, complex lines, which, according to Lemma~\ref{3}, is equivalent
to $\sigma_p(A,A^*)$ being the union of countably many, locally finite, complex
lines.
\end{proof}

\section{Complete commutativity}

In this section we determine when two operators $A$ and $B$ completely commute,
namely, $AB=BA$ and $AB^*=B^*A$. It is clear that $A$ and $B$ completely commute if
and only if $B$ commutes with both $A$ and $A^*$, so complete commutativity is a
symmetric relation. It is a well-known theorem of Fuglede \cite{F} that if $B$ is
normal, then $A$ commutes with $B$ if and only if $A$ commutes with $B^*$. Therefore,
for normal operators, complete commutativity is the same as commutativity.

\begin{thm}
Suppose $A$ and $B$ are both compact. Then they are normal and commute if and only
if the joint point spectrum $\sigma_p(A,A^*,B,B^*)$ is the union of countably many
complex lines in $\C^4$.
\label{17}
\end{thm}

\begin{proof}
Consider the self-ajoint operators
$$A_1=A+A^*,\quad A_2=i(A-A^*),\quad A_3=B+B^*,\quad A_4=i(B-B^*).$$
It is easy to check that the normal operators $A$ and $B$ completely commute if and only
if the operators in $\{A_1,A_2,A_3,A_4\}$ pairwise commute, which, by Theorem~\ref{11}, is
equivalent to $\sigma_p(A_1,A_2,A_3,A_4)$ being the union of countably many, locally finite,
complex hyperplanes. This, according to Lemma~\ref{3}, is equivalent to $\sigma_p(A,A^*,B,B^*)$
being the union of countably many, locally finite, complex hyperplanes.
\end{proof}

\begin{thm}
Suppose that $A$ and $B$ are both compact. Then they commute completely if and only if
each of the four joint spectra $\sigma_p(A\pm A^*,B\pm B^*)$ is the union of countably
many, locally finite, complex lines in $\C^2$.
\label{18}
\end{thm}

\begin{proof}
If each of the four sets $\sigma_p(A\pm A^*,B\pm B^*)$ is the union of countably many, locally
finite, complex lines, then by Lemma~\ref{3} and Theorem~\ref{11}, all four pairs of operators
commute. Thus, we have
\begin{eqnarray}
AB+AB^*+A^*B+A^*B^*=BA+B^*A+BA^*+B^*A^* \label{eq17} \\
AB+A^*B-AB^*-A^*B^*=BA-B^*A+BA^*-B^*A^* \label{eq18} \\
AB+AB^*-A^*B-A^*B^*=BA+B^*A-BA^*-B^*A^* \label{eq19} \\
AB-A^*B-AB^*+A^*B^*=BA-B^*A-BA^*+B^*A^* \label{eq20}
\end{eqnarray}
Adding equations (\ref{eq17})-(\ref{eq20}) gives $AB=BA$.
Adding (\ref{eq17}) and (\ref{eq18}), we obtain $AB+A^*B=BA+BA^*$.
Thus, $A$ and $B$ commute completely.

The opposite direction follows from Theorem \ref{11} and Lemma~\ref{3} as well.
\end{proof}

\section{Further remarks and extensions}

We conjecture that Corollary D stated in the introduction can be strengthened as follows: If
$\bA=(A_1,\cdots,A_n)$ is a tuple of compact and normal operators, then the operators in
$\bA$ pairwise commute if and only if $\sigma_p(\bA)$ is the union of countably many, locally
finite, complex hyperplanes in $\cn$. Note that we have already shown this for matrices. But
the proof for matrices depends on the the determinant function and the compactness of
the unit sphere in $\C^N$. The determinant function can be extended to operators of the form
$z_1A_1+\cdots+z_nA_n+I$, where each $A_k$ is in the trace class. However, the unit sphere in
an infinite dimensional Hilbert space is only compact in the weak topology, and this does not
appear enough for our purposes. 

Our focus here is on the linear structure in the joint point spectrum $\sigma_p(\bA)$ of a
tuple of compact operators. Some of our ideas and techniques can be applied to certain other
situations. For example, some of our results hold for certain operators with discrete spectrum,
although the case of continuous spectrum seems to be completely different. Also, we have
obtained some partial results about the commutativity of operators based on certain nonlinear
geometric properties of $\sigma_p(\bA)$. We will discuss several related problems and results
in subsequent papers, and we hope that this paper will serve as a catalyst for further research
in this field.

In \cite{R1} Ricker proved a beautiful theorem stating that an \(n\)-tuple \(\bA=(A_1,\cdots,A_k)\)
of self-adjoint matrices is mutually commuting if and only if the following matrix-valued distribution
\begin{equation}\label{Weyl}
{\mathcal T}_{\bA}f=\left ( \frac{1}{2\pi}\right ) ^{n/2}\int_{\R^n}
e^{i\langle w,A\rangle}\hat{f}(w)dw, \ f\in S(\R^n),
\end{equation}
has order zero. Here, as usual, \(S(\R^n)\) stands for the Schwartz space of complex-valued, rapidly
decreasing functions on \(\R^n\), and \(\hat{f}\) is the Fourier transform of \(f\). Ricker further
posted the problem of whether a similar result holds for an \(n\)-tuple of self-adjoint operators
acting on a Hilbert space \(H\), and commented that the technique in \cite{R1} was purely
finite-dimensional. This problem seems to be still open.

Our Theorem \ref{15} also deals with commutativity of an \(n\)-tuple of matrices (from a slightly
wider class of normal matrices). Our technique is essentially infinite-dimensional. It would be
interesting to find out whether there is a connection between the geometry of the projective
joint spectrum of an \(n\)-tuple of compact self-adjoint operators and the order of
distribution in (\ref{Weyl}). In particular, we wonder if it is possible to tackle Ricker's problem
for compact operators from this angle.

Finally, we use $2\times2$ matrices to demonstrate that the normality assumption in
Theorem~\ref{15} is necessary. In fact, if we take
$$A=\begin{pmatrix} 1&0\cr 0&2\end{pmatrix},\qquad
B=\begin{pmatrix} 3&0\cr 4&5\end{pmatrix}.$$
Then
$$\det(I+zA+wB)=(1+z+3w)(1+2z+5w)$$
is completely reducible. But these two matrices do not commute.

\end{document}